\documentclass[11pt, a4paper]{article}
\usepackage[english]{babel}
\usepackage{a4wide}
\usepackage[utf8]{inputenc}
\usepackage{amssymb,amsmath,amsthm}

\newtheorem{theorem}{Theorem}

\newtheorem{lemma}{Lemma}

\begin{document}
\title{A new rate of convergence estimate for homogeneous discrete-time nonlinear Markov chains}

\author{A.A. Shchegolev \footnote{National Research University Higher School of Economics, Moscow, Russia, e-mail:\newline ashchegolev@hse.ru.}}

\maketitle

\abstract{In the paper, we study a new rate of convergence estimate for homogeneous discrete-time nonlinear Markov chains based on the Markov-Dobrushin condition. This result generalizes the convergence estimates for any positive number of transition steps. An example of a class such a process provided indicates that such types of estimates considering several transition steps may be applicable when one transition can not guarantee any convergence. Moreover, a better estimate can be obtained for a higher number of transitions steps. A law of large numbers is presented for a class of ergodic nonlinear Markov chains with finite state space that may serve as a basis for nonparametric estimation and other statistics.\newline
Keywords: 60J05; Nonlinear Markov chains; ergodicity; rate of convergence; law of large numbers}

\section{Introduction} 
\label{S:1}
The article considers estimating the rate of convergence to an invariant measure for a specific class of nonlinear Markov chains. The work develops the approach presented in the article \cite{shchegolev:2021}, which initially adopts the idea of \cite{butkovsky:2014} for two-step transition kernels. In these works the Markov-Dobrushin type condition plays an essential role
\begin{align}
\sup_{\mu,\nu\in\mathcal{P}(E)}\|P_\mu(x,\cdot) - P_\nu(y,\cdot)\|_{TV} \le 2(1-\alpha),\quad 0 < \alpha < 1, \quad x,y\in E,
\label{NonlinearDobrushinAnalog}
\end{align}
as well as the following restrictive condition
\begin{align}
\|P_\mu(x,\cdot) - P_\nu(x,\cdot)\|_{TV} \le \lambda \|\mu - \nu\|_{TV},\quad \lambda\in [0,\alpha],\quad x\in E,\quad \mu,\nu \in \mathcal{P}(E).
\label{ButkovskyLambda}
\end{align}

The current article further develops the approach from \cite{shchegolev:2021}, checking the ergodicity conditions for discrete-time homogeneous nonlinear Markov chains by taking $k$-step transition probabilities, where $k$ is an arbitrary positive number. An example of such a process is presented, illustrating that the estimate taken on multiple steps is applicable while the one-step estimate is not. This example also provided the rate of convergence calculation for $k=2$ and $k=3$, indicating that estimates for the higher number of steps can be slightly better. Finally, the law of large numbers is deduced for the nonlinear Markov chains for the case of finite state space.

The rest of the article is organized as follows. In Section 2, the results generalizing estimates for $k$-step are provided.
Section 3 provides an example of such a discrete-time homogeneous nonlinear Markov chain with an explicit rate of convergence calculation. The final section establishes the law of large numbers for nonlinear Markov chains that satisfy the conditions for the existence and uniqueness of invariant measure.

\section{Rate of convergence estimate for k-step transitions}
\label{S:2}
Let the process $\left(X_n^\mu\right)_{n\in\mathbb{Z}_{+}}$ be a nonlinear Markov chain defined on the finite state space $(E, \mathcal{E})$, initial distribution $\mu = \text{Law}\left(X_0^\mu\right)$, $\mu \in \mathcal{P}(E)$ and transition probabilities\newline $P_{\mu_n}(x, B) = \mathbb{P}_{\mu_n}\left(X_{n+1}^\mu \in B | X_n^\mu = x\right)$,
where $x \in E$, $B \in \mathcal{E}$, $n \in \mathbb{Z}_{+}$ and $\mu_n := \text{Law}(X_n^\mu)$. 

Let $\mu, \nu \in \mathcal{P}(E)$, then the total variation distance between two probability measures may be defined as follows:
\vspace{-0.5em}
\begin{equation*}
\|\mu - \nu\|_{TV} = 2\sup_{A\in \mathcal{E}}|\mu(A) - \nu(A)| = \int_E |\mu(dx) - \nu(dx)|.
\vspace{-0.25em}\end{equation*}

According to the \cite{butkovsky:2014} results, the nonlinear Markov chain is a uniformly ergodic process, and the existence of a unique invariant measure $\pi$ is guaranteed if the conditions \eqref{NonlinearDobrushinAnalog} and \eqref{ButkovskyLambda} are satisfied.

Then in case when $\lambda < \alpha$ we have an exponential rate of convergence
\vspace{-0.5em}
\begin{equation*}
\|\mu_n - \pi\|_{TV} \le 2(1-(\alpha - \lambda))^n,\quad n\in\mathbb{Z}_+,
\vspace{-0.25em}\end{equation*}
while in case $\lambda = \alpha$ there is a linear convergence
\vspace{-0.5em}
\begin{equation*}
\|\mu_n - \pi\|_{TV} \le \frac{2}{\lambda n},\quad n\in\mathbb{Z}_+.
\vspace{-0.25em}\end{equation*}
Otherwise, when $\lambda > \alpha$, there may be either an infinite number of invariant measures or no invariant measure at all.

The paper \cite{shchegolev:2021} proposes a generalization of this approach to a two-step transition probability kernel. The current article elaborates the idea of this paper and mainly follows along the same lines by using $k$-step transition probabilities, $k > 0$.

\begin{theorem}[Existence and uniqueness of an invariant measure]
Let the process $X$ have a $k$-step transition probabilities $Q_{\mu_n}(x, A): = P_{\mu_n}(X^\mu_{n+k}\in A | X^\mu_{n} = x)$ and satisfies the following conditions:
\begin{align}
\sup_{\mu,\nu\in\mathcal{P}(E)}\|Q_{\mu}(x,\cdot) - Q_{\nu}(y,\cdot)\|_{TV} \le 2(1-\alpha_k),
\label{NonlinearDobrushinAnalog2}
\end{align}
where $0 < \alpha_k < 1, \quad x,y\in E$,
\begin{align}
\|Q_{\mu}(x,\cdot) - Q_{\nu}(x,\cdot)\|_{TV} \le \lambda_k \|\mu - \nu\|_{TV},
\label{Lambda2}
\end{align}
where $\lambda_k \in [0,\alpha_k],\quad x\in E,\quad \mu,\nu \in \mathcal{P}(E)$,
\begin{align}
\|P_{\mu}(x,\cdot) - P_{\nu}(x,\cdot)\|_{TV} \le \lambda_1\|\mu - \nu\|_{TV},\,\,\lambda_1 < \infty.
\label{Lambda1}
\end{align}
Then the process $X$ has a unique invariant measure $\pi$ and for any probability measure $\mu \in \mathcal{P}(E)$ the convergence is true:
\begin{flalign}\label{pi_lambda<alpha}
\|\mu_n - \pi\|_{TV} \le \|\mu_0 - \pi\|_{TV}(1-\alpha_k  + \lambda_k)^{[n/k]}(1 + \lambda_1)^{[n \bmod k]}.
\end{flalign}
while in case $\lambda_k = \alpha_k$
\begin{align}
\|\mu_n - \pi\|_{TV} \le \frac{\|\mu_0 - \pi\|_{TV}}{2 + \lambda_k n\|\mu_0 - \pi\|_{TV}}(1 + \lambda_1)^{[n\bmod k]}.
\label{pi_lambda=alpha}
\end{align}
\label{th1}
\end{theorem}
\vspace{-0.5em}
In order to prove this theorem, we need the following theorem on the convergence of any two initial probability measures for the process under consideration.

\begin{theorem}
Let the process $X$ have a $k$-step transition probabilities $Q_\mu(x,B)$ and satisfies conditions \eqref{NonlinearDobrushinAnalog2}, \eqref{Lambda2} and \eqref{Lambda1} of the Theorem \ref{th1}.
Then for any pair of probability measures $\mu, \nu \in \mathcal{P}(E)$ the following convergence is true:
{\small
\begin{flalign}\label{lambda<alpha}
\|\mu_n - \nu_n\|_{TV} \le  \|\mu_0 - \nu_0\|_{TV}(1-\alpha_k  + \lambda_k)^{[n/k]}(1 + \lambda_1)^{[n \bmod k]}.
\end{flalign}}
while in case $\lambda_k = \alpha_k$
\begin{align}
\|\mu_n - \nu_n\|_{TV} \le \frac{\|\mu_0 - \nu_0\|_{TV}}{2 + \lambda_k n\|\mu_0 - \nu_0\|_{TV}}(1 + \lambda_1)^{[n\bmod k]}.
\label{lambda=alpha}
\end{align}
\label{th2}
\end{theorem}

Let us prove the theorem \ref{th2}.
\begin{proof}
Let $P: E \times \mathcal{E} \rightarrow [0,1]$ be a transition kernel, $\varphi: E\rightarrow \mathbb{R}$ be a measurable function and probability measure $\mu \in \mathcal{P}(E)$; denote
$\mu P:= \int_EP(x,dt)\mu(dx)$; in case when $P$ depends on measure $\mu$, we have: $\mu_1(\mu):=\mu P_\mu:= \int_EP_\mu(x,dt)\mu(dx)$, then $k$-step transition probability kernel
\vspace{-0.5em}
\begin{equation*}
Q_\mu(x,dy) = \int P_\mu(x,dx_1)P_{\mu_1(\mu)}(x_1,dx_2)\dots P_{\mu_{k-1}(\mu)}(x_{k-1},dy).
\vspace{-0.25em}\end{equation*}

Consider the total variation distance between measures after applying the $k$-step transition kernel. For any probability measure $\mu, \nu \in \mathcal{P}(E)$ denote $$d\eta = ((d\mu / d\nu) \wedge 1)d\nu$$ and apply the triangle inequality to obtain
\begin{align*}
\|\mu Q_{\mu}- \nu  Q_{\nu}\|_{TV} =
\|(\eta + (\mu - \eta)) Q_{\mu}- (\eta - (\nu - \eta))  Q_{\nu}\|_{TV} =\\
= \int_E |\eta Q_{\mu}(dx) + (\mu - \eta)Q_{\mu}(dx)- \eta Q_{\nu}(dx) - (\nu - \eta)Q_{\nu}(dx)| \le\\
\le \int_E |\eta Q_{\mu}(dx) - \eta Q_{\nu}(dx)|+ \int_E |(\mu - \eta)Q_{\mu}(dx) - (\nu - \eta)Q_{\nu}(dx)| \le \\
\le \|\eta Q_{\mu} - \eta Q_{\nu}\|_{TV} + \|(\mu - \eta)Q_{\mu} - (\nu - \eta)Q_{\nu}\|_{TV}.
\end{align*}

Consider the first term, applying Jensen's inequality and \eqref{Lambda2} to it, and using the following fact:
 $\eta(E) = 1 - \|\mu-\nu\|_{TV}/2$. We get 
{\small
\begin{align*}
\|\eta Q_{\mu} - \eta Q_{\nu}\|_{TV} &= \int_E \left|\int_EQ_{\mu}(x,dy)\eta(dx) - \int_EQ_{\nu}(x,dy)\eta(dx)\right| \le \\
&\le \int_E\int_E|Q_{\mu}(x,dy) - Q_{\nu}(x,dy)|\eta(dx) \le \\ 
&\le \lambda_k \|\mu-\nu\|_{TV}\left(1 - \frac{1}{2}\|\mu-\nu\|_{TV}\right).
\end{align*}
}
Then the second term
{\small
\begin{align*}
\|(\mu - \eta)Q_{\mu} - (\nu - \eta)Q_{\nu}\|_{TV} = \int_E |(\mu - \eta)Q_{\mu}(dx) - (\nu - \eta)Q_{\nu}(dx)| =\\
=
\int_E \left|\int_EQ_{\mu}(x,dy)(\mu - \eta)(dx) - \int_EQ_{\nu}(x',dy)(\nu - \eta)(dx')\right|.
\end{align*}
}

Remind that $\mu_n = \text{Law}(X_n^\mu)$, $\nu_n = \text{Law}(X_n^\nu)$, denote $p_0 = \|\mu_0 - \nu_0\|_{TV}/2$, assuming $p_0>0$ (if $p_0=0$, then $p_k=0$ etc.).

Estimating the expression $\|\mu_k - \nu_k\|_{TV}$ from above.
\begin{align*}
\|\mu_k - \nu_k\|_{TV} \le \lambda_k \|\mu_0 -  \nu_0\|_{TV} \left(1- \frac12 \|\mu_0 - \nu_0\|_{TV}\right)
\\
+ p_0 \int \left|\int Q_{\mu}(x, dy)\frac{(\mu_0-\eta_0)(dx)} {p_0}- \int Q_{\nu}(x', dy)\frac{(\nu_0-\eta_0)(dx')} {p_0}\right|
\\
=  2p_0 \lambda_k (1-p_0)  + p_0\int
\left|\iint (Q_{\mu}(x, dy) - Q_{\nu}(x', dy)) \frac{(\mu_0-\eta_0)(dx)} {p_0}  \frac{(\nu_0-\eta_0)(dx')} {p_0}\right|
\\
\le \,  2p_0 \lambda_k (1-p_0)  + p_0\iiint \left| Q_{\mu}(x, dy)- Q_{\nu}(x', dy)\right|\frac{(\mu_0-\eta_0)(dx)} {p_0}\frac{(\nu_0-\eta_0)(dx')} {p_0}
\\
\le 2p_0 \lambda_k (1-p_0)  +
2(1-\alpha_k)p_0 \iint \frac{(\mu_0-\eta_0)(dx)} {p_0} \frac{(\nu_0-\eta_0)(dx')} {p_0}
\\
= 2p_0 \lambda_k (1-p_0)  +
2p_0(1-\alpha_k) =
2p_0 (\lambda_k - \lambda_k p_0 + 1 - \alpha_k).
\end{align*}
If $\lambda_k < \alpha_k$, we obtain
\vspace{-0.5em}
\begin{equation*}
\|\mu_k - \nu_k\|_{TV} \le \|\mu_0-\nu_0\|_{TV}(1 - \alpha_k + \lambda_k),
\vspace{-0.25em}\end{equation*}
while in the case $\lambda_k = \alpha_k$ we get
\vspace{-0.5em}
\begin{equation*}
\|\mu_k - \nu_k\|_{TV} \le 2p_0(1 - \lambda_kp_0),
\vspace{-0.25em}\end{equation*}
or
\vspace{-0.5em}
\begin{equation*}
p_k \le p_0(1 - \lambda_kp_0).
\vspace{-0.25em}\end{equation*}

In case $kn+i$, $0<i<k$ we have:
\begin{align*}
\|\mu_{kn+i} - \nu_{kn+i}\|_{TV} \le \lambda_i \|\mu_{kn} -  \nu_{kn}\|_{TV} \left(1- \frac12 \|\mu_{kn} - \nu_{kn}\|_{TV}\right)
\\
+ p_{kn} \int \left|\int P_{\mu}(x, dy)\frac{(\mu_{2n}-\eta_{2n})(dx)} {p_{2n}}- \int P_{\nu}(x', dy)\frac{(\nu_{2n}-\eta_{2n})(dx')} {p_{2n}}\right|
\\
=  2p_{2n} \lambda_i (1-p_{2n})  + p_{2n}\int
\left|\iint (P_{\mu}(x, dy) - P_{\nu}(x', dy)) \frac{(\mu_{2n}-\eta_{2n})(dx)} {p_{2n}}  \frac{(\nu_{2n}-\eta_{2n})(dx')} {p_{2n}}\right|
\\
\le \,  2p_{2n} \lambda_i (1-p_{2n})  + p_{2n}\iiint \left| P_{\mu}(x, dy)- P_{\nu}(x', dy)\right|\frac{(\mu_{2n}-\eta_{2n})(dx)} {p_{2n}}\frac{(\nu_{2n}-\eta_{2n})(dx')} {p_{2n}}
\\
\le 2p_{2n} (\lambda_i (1-p_{2n})  + 1) = (1 + \lambda_i)\|\mu_{2n} -  \nu_{2n}\|_{TV}.
\end{align*}
Hence,
\begin{align}
\|\mu_{2n+1} - \nu_{2n+1}\|_{TV} \le (1 + \lambda_1)\|\mu_{2n} -  \nu_{2n}\|_{TV}.
\label{OneStep1lambda}
\end{align}

Thus, iterating the estimate for $\lambda_k < \alpha_k$, we obtain by induction
\vspace{-0.5em}
\begin{equation*}
\|\mu_n - \nu_n\|_{TV} \le  \|\mu_0 - \nu_0\|_{TV}(1-\alpha_k  + \lambda_k)^{[n/k]}(1 + \lambda_1)^{[n \bmod k]}.
\vspace{-0.25em}\end{equation*}

In case $\alpha_k = \lambda_k$ we apply the following lemma.
\begin{lemma}
\label{lemma}
Let $a_0, a_1, \dots$ be a sequence of positive numbers. Assume that $0 < a_0 \le 1$ and the following estimate is true
\vspace{-0.5em}
\begin{equation*}
a_{n+1} \le a_n(1-\psi(a_n)),\quad n\in\mathbb{Z}_+,
\vspace{-0.25em}\end{equation*}
where $\psi: [0,\infty)\rightarrow [0,1]$ is a continuous non-decreasing function with $\psi(0) = 0$ and $\psi(x) > 0$ as $x>0$. Then
\vspace{-0.5em}
\begin{equation*}
a_n \le g^{-1}(n)
\vspace{-0.25em}\end{equation*}
for all $n\in\mathbb{Z}_+$, where
\vspace{-0.5em}
\begin{equation*}
g(x) = \int_x^{a_0}\frac{dt}{t\psi(t)},\quad 0 < x \le 1.
\vspace{-0.25em}\end{equation*}
\end{lemma}

This lemma and its proof are given in a slightly different version in \cite[Lemma~1.4.2]{butkovsky:2013}. Since we use a slightly modified version with a different upper limit in the integral, the proof is presented for the reader's convenience, even though it coincides with the source.

\begin{proof}
Notice that function $g^{-1}$ exists, as $g$ is unbounded, non-negative and strictly decreasing. Since $\psi(x) > 0$, then $a_{n+1} \le a_n$ for any $n \in \mathbb{N}$. Therefore we can find $s\in[a_{n+1}, a_n]$ such that $$g(a_{n+1})-g(a_{n}) = g'(s)(a_{n+1}-a_n) = -\frac{a_{n+1}- a_n}{s\psi(s)} \ge \frac{a_n\psi(a_n)}{s\psi(s)} \ge 1.$$
Thus, $g(a_n) \ge n$ and $a_n \le g^{-1}(n)$.
\end{proof}

Applying the lemma \ref{lemma} for $a_{n+k}$ and $\psi(t) = \lambda_k t$ we obtain
\vspace{-0.5em}
\begin{equation*}
p_n \le g^{-1}(n) = \frac{1}{\lambda_k n + \frac{1}{p_0}} = \frac{p_0}{1+p_0\lambda_k n}.
\vspace{-0.25em}\end{equation*}

Generalizing the result for arbitrary $n$ using \eqref{OneStep1lambda} we get
\vspace{-0.5em}
\begin{equation*}
\|\mu_n - \nu_n\|_{TV} \le \frac{\|\mu_0 - \nu_0\|_{TV}}{2 + \lambda_k n\|\mu_0 - \nu_0\|_{TV}}(1 + \lambda_1)^{[n\bmod k]}.
\vspace{-0.25em}\end{equation*}
\end{proof}
\vspace{-0.5em}

Next, we proceed to the proof of Theorem \ref{th1}.
\begin{proof}
Consider a sequence of probability measures $\left(\mu_n\right)_{n\in\mathbb{N}}$. According to the Theorem \ref{th2}, by virtue of \eqref{lambda<alpha} and \eqref{lambda=alpha}, for any $m,n \in \mathbb{N}$
\vspace{-0.5em}
\begin{equation*}
\|\mu_n - \mu_{n+m}\|_{TV} \le \frac{\|\mu_0 - \mu_m\|_{TV}}{2 + \lambda_k n\|\mu_0 - \mu_m\|_{TV}}(1 + \lambda_1)^{[n \bmod k]}.
\vspace{-0.25em}\end{equation*}
Then $\left(\mu_n\right)_{n\in\mathbb{N}}$ is a Cauchy sequence in complete metric space $(\mathcal{P}(E), \|\cdot\|_{TV})$ and we can find $\pi\in\mathcal{P}(E)$ such that $\lim_{n\rightarrow \infty}\|\mu_n-\pi\|_{TV} = 0$.

Let us show that the limiting measure $\pi$ is invariant. For this, we can use the triangle inequality and the condition \eqref{OneStep1lambda} as $n\rightarrow \infty$:
\vspace{-0.5em}
\begin{equation*}
\|\pi P_\pi - \mu_{n+1}\|_{TV} = \|\pi P_\pi - \mu_{n}P_{\mu_{n}}\|_{TV} \le (1+\lambda_1)\|\pi-\mu_n\|_{TV}\rightarrow 0,
\vspace{-0.25em}\end{equation*}
while $\mu_{n+1} \rightarrow \pi$, we have
\vspace{-0.5em}
\begin{equation*}
\|\pi P_\pi - \pi\|_{TV} \le \|\pi P_\pi - \mu_{n+1}\|_{TV} + \|\mu_{n+1} - \pi\|_{TV} \rightarrow 0.
\vspace{-0.25em}\end{equation*}
From where we get $\pi = \pi P_\pi$.

To prove the uniqueness of the invariant measure $\pi$, assume that $\nu\in\mathcal{P}(E)$ is such that $\nu \ne \pi$ and $\nu = \nu P_\nu$, then iteratively applying the results \eqref{lambda<alpha} and \eqref{lambda=alpha}, for a sufficiently large $n$ we obtain a contradiction
\vspace{-0.5em}
\begin{equation*}
\|\nu-\pi\|_{TV} = \|\nu Q_\nu-\pi Q_\pi\|_{TV} = \|\nu Q^n_\nu-\pi Q^n_\pi\|_{TV} < \|\nu-\pi\|_{TV}.
\vspace{-0.25em}\end{equation*}
Thus, the process $(X_n^{\mu})_{n\in\mathbb{Z}}$ has a unique invariant measure $\pi$.
\end{proof}

In section \ref{S:3} we show an example of a discrete-time nonlinear Markov chain satisfying such ergodic conditions.

\section{An example of ergodic nonlinear Markov chain}
\label{S:3}
Let us show that this result can be used in cases where the one-step estimate is inapplicable, and, at the same time, violation of \cite{butkovsky:2014} conditions does not prevent exponential convergence for some nonlinear Markov chains. Consider the following discrete nonlinear Markov chain $X_n^\mu$ with state space $(E, \mathcal{E}) = \left(\{1, 2, 3, 4\}, 2^{\{1, 2, 3, 4\}}\right)$, the initial distribution $\mu_0$ and the transition probability matrix $P_{\mu_0}(i, j)$, defined as follows:
\vspace{-0.5em}
\begin{equation*}
\mu_0 = \begin{pmatrix}\mu(\{1\}) & \mu(\{2\})& \mu(\{3\})& \mu(\{4\})\end{pmatrix},
\vspace{-0.25em}\end{equation*}
\vspace{-0.5em}
\begin{equation*}
P_{\mu_0}(i, j) =
\begin{pmatrix}
0 & \gamma  \mu(\{1\}) & 0.5 - \gamma  \mu(\{1\}) & 0.5\\
0.5 & 0.5 & 0 & 0 \\
0.5 & 0 & 0.5 & 0 \\
0 & 0.5 & 0 & 0.5
\end{pmatrix},
\vspace{-0.25em}\end{equation*}
where $0 \le \gamma \le 0.5$.

We can notice that for a given process the conditions \eqref{NonlinearDobrushinAnalog} and \eqref{ButkovskyLambda} do not guarantee convergence to an invariant measure, since $\lambda > \alpha$, since $\alpha = 0$ and $\lambda = \gamma$.

However, if we consider the corresponding three-step transition probabilities matrix $Q_{\mu_0}(i,j)$,

\begin{align*}
&Q_{\mu_0}(i,j) = 
&\begin{pmatrix}
0.25 & 0.0625 (4 + \gamma + 4\gamma\nu_1) & 0.0625 (4 - \gamma - 4\gamma \nu_1) & 0.25 \\
0.25 & 0.0625 (4 + \gamma + 2\gamma (\nu_2 + \nu_3)) & 0.0625 (4 - \gamma - 2\gamma (\nu_2 + \nu_3)) & 0.25 \\
0.25 & 0.0625 (2 + \gamma + 2\gamma (\nu_2 + \nu_3)) & 0.0625 (6 - \gamma - 2\gamma (\nu_2 + \nu_3)) & 0.25 \\
0.25 & 0.0625 (6 + \gamma) & 0.0625 (2 - \gamma) & 0.25
\end{pmatrix},
\end{align*}
we may obtain the following result. We have $\lambda_3 < \alpha_3$, as $\lambda_3 = \gamma/4$, while $\alpha_3$ reaches its minimum for a pair of states $\{3, 4\}$ with the value in range $[0.75; 0.75 + 0.125 \gamma]$, thus $\alpha_3 = 0.75$.
Thus, the proposed estimate can guarantee exponential convergence in some cases when the existing \cite{butkovsky:2014} result does not work. And we may also compare the rate of convergence on this example by taking estimates for two \cite{shchegolev:2021} and three steps. We may consider the part of estimate $(1-\alpha_k  + \lambda_k)^{[n/k]}(1 + \lambda_1)^{[n \bmod k]}$ for the case $\alpha_k < \lambda_k$, since we have straightforward results in case $\alpha_k=\lambda_k$. In Table \ref{tab:Table1} we calculate the rate of convergence for $k=2$ and $k=3$ in limit values of the parameter $\gamma\in[0,1/2]$. Here we omit the multiplier $(1+\lambda_k)^{[n \bmod k]}$ for the case $\gamma=1/2$, since it does not impact much the convergence itself.

\begin{table} [!ht]
\caption{Rate of convergence calculation, $(1-\alpha_k  + \lambda_k)^{[n/k]}$.}
\begin{tabular}{|l|c|c|}
\hline
& $k=2$ & $k=3$ \\
\hline
$\gamma=0$& $0.707107^n$ & $0.629961^n$ \\
\hline
$\gamma=1/2$ &  $0.866025^n$ & $0.721125^n$\\
\hline
\end{tabular}
\centering
\label{tab:Table1}
\end{table}

We can see that the result obtained using three steps tend to be better given that the remainders $n \bmod k$ are the same.

\section{Law of large numbers}
This section presents a particular result of the law of large numbers for nonlinear Markov chains. In literature, there exist similar results, for example, \cite{delmoral:2011}, still, the model under consideration is quite different. However, such a result can be helpful in nonparametric estimation and serves as the basis for further statistics.

\begin{theorem}
Let $X_k^\mu$ be a nonlinear Markov chain defined on a measurable finite state space $(E, \mathcal{E})$ that satisfies the conditions of Theorem \ref{th1} and $f,g: E\rightarrow\mathbb{R}$ be a measurable bounded continuous functions. The function $f$ is also non-negative and should have uniform modulus of continuity.  Denote $S_n=\sum_{k=0}^{n-1}g(X_k^\mu)$, then the sequence $S_n$ satisfies the law of large numbers, such as, $n\rightarrow\infty$
\begin{flalign}
\left(\frac{S_n}{n}\right)\overset{\mathbb{P}} \longrightarrow f(\mathbb{E}[g(X^\pi_0)]),
\label{weaklln}
\end{flalign}
where $\mathbb{E}[g(X^\pi_n)]=\int g(x)\pi(dx)$, $X_k^\pi$ is a copy of $X_k^\mu$ with initial distribution equal to the invariant measure $\pi$.
\label{th3}
\end{theorem}

\begin{proof}
First, we can notice that $X_k^\pi$, which is, in fact, a homogeneous ``linear'' Markov chain that does not depend on the measure, and it is also $\beta$-mixing due to the \eqref{NonlinearDobrushinAnalog2}. Denote  $\xi_k = g(X_k^\pi) - \mathbb{E}[g(X^\pi_k)]$ and consider the sequence of sums $\tilde S^\pi_n=\sum_{k=0}^{n-1}\xi_k$.

Then, according to the results of Ibragimov and Linnik \cite[Theorem~18.5.3]{ibragimov:1965}, we have that $\mathbb{E}[\xi_k|^{2+\delta}]<\infty$ for some $\delta > 0$ and $\sum_{n=1}^{\infty}\alpha(n)^{\delta/(2+\delta)} <\infty$, where $\alpha(n)$ is a mixing coefficient. The variance of $\tilde S^\pi_n$ is finite, $\sigma^2 = \text{Var}(\tilde S^\pi_n) = \sum_{j=-\infty}^{\infty}\text{cov}(\xi_0,\xi_j) < \infty$. Then, the sequence $\tilde S^\pi_n$ satisfies the central limit theorem, $\tilde S^\pi_n \sim \mathcal{N}(0, \sigma^2/n)$. The case $\sigma^2 = 0$ can be interpreted as a degenerate Gaussian distribution, or Dirac delta distribution concentrated at the point $0$. 

Therefore, the process $S^\pi_n$ also satisfies the central limit theorem, implying the weak law of large numbers,
\begin{flalign}
S^\pi_n/n \overset{\mathbb{P}}{\longrightarrow} \mathbb{E}[g(X^\pi_0)].
\label{stationarylln}
\end{flalign}

Next, we need to show that $S_n/n - S^\pi_n/n$ converges to $0$.
In order to prove this convergence, a coupling construction technique for nonlinear Markov chains similar to the method presented in \cite{veretennikov:2013,veretennikov:2020,veretennikov:2021} can be developed. 
 
Alternatively, we can use the estimates for convergence from Theorem \ref{th1} to show that for $n \rightarrow \infty$, $\forall f\in \mathcal{C}_b$, $f \ge 0$, $\mathbb{E}\left[f(S_n/n)\right] \rightarrow f\left(\mathbb{E}\left[g(X_0^\pi)\right]\right)$.

Let 
$$\mathbb{E}[f(S_{n}/n)] = \iint\dots\int f\left(\frac{x_0+x_1+\dots+x_{n-1}}{n}\right)\mu_0(x_0)P^{\mu_0}_{x_0,dx_1}P^{\mu_1}_{x_1,dx_2}\dots P^{\mu_{n-2}}_{x_{n-2},dx_{n-1}}.$$

From the Theorem \ref{th1}, we get that $\|\mu_n-\pi\|_{TV} \le 2e^{-Cn}$, where $C$ is a constant that does not depend on $n$. Then assume that all elements of $P^{\pi}_{x,y}$ are bounded away from zero uniformly with respect to the measure,
$$
-\frac{\|\mu_n-\pi\|_{TV}}{2P^{\pi}_{x,y}} \le \frac{P^{\mu_n}_{x,y} - P^{\pi}_{x,y}}{P^{\pi}_{x,y}} \le \frac{\|\mu_n-\pi\|_{TV}}{2P^{\pi}_{x,y}},
$$
hence we deduce
$$
1 - e^{\ln K - Cn} \le 1 - \frac{e^{-Cn}}{P^\pi_{x,y}} \le \frac{P^{\mu_n}_{x,y}}{P^{\pi}_{x,y}} \le 1 + \frac{e^{-Cn}}{P^\pi_{x,y}} \le 1 + e^{\ln K - Cn}.
$$

Then, we can assume that starting from some moment $n_0+1$ such that $n_0 \ll n$, the transition kernels $P^{\mu_n}_{x,y}$ are close to the invariant transition kernel $P^{\pi}_{x,y}$. We may define a transition kernel for $n_0$ steps, $Q^{\mu_0, n_0}_{x_0, dx_{n_0}}$, such that
$$
\mu_{n_0} = \mu_0 Q^{\mu_0, n_0}_{x_{0}, dx_{n_0}}(dx_{n_0}) = \int \mu_0(dx_0)Q^{\mu_0, n_0}(x_0, dx_{n_0}).
$$

Then we can denote $\rho_f$ the modulus of continuity of function $f$, $\left|\sum_{i=0}^{n_0}x_i/n\right| \le \delta$ and rewrite

{\scriptsize
\begin{flalign*}
\mathbb{E}[f(S_{n}/n)] =& \iint\dots\int f\left(\frac{x_0+\dots+x_{n_0}}{n} + \frac{x_{n_0+1}+\dots+x_{n-1}}{n}\right)\mu_0Q^{\mu_0, n_0}(dx_{n_0})\prod_{k=n_0}^{n-2}P^{\mu_k}_{x_k,dx_{k+1}}\\
\le&
\rho_f(\delta) + \iint\dots\int f\left(\frac{x_{n_0+1}+\dots+x_{n-1}}{n}\right)
\left(1+e^{\ln K - C(n_0-1)}\right)\pi(dx_{n_0})\\
&\times
\prod_{k=n_0}^{n-2}\left(1+e^{\ln K - Ck}\right)P^{\pi}_{x_k,dx_{k+1}}\\
\le& \rho_f(\delta) + \prod_{k=n_0-1}^{n-2}\left(1+e^{\ln K - Ck}\right)\\&\times
\left(\iint\dots\int f\left(\frac{\sum_{i=n_0+1}^{n-1}x_i}{n}\right)
\pi(dx_{n_0})\prod_{k=n_0+1}^{n-2}P^{\pi}_{x_k,dx_{k+1}}\right).
\end{flalign*}}

Then, for the invariant distribution we have
\begin{flalign*}
\mathbb{E}[f(S^\pi_{n}/n)] =& \iint\dots\int f\left(\frac{x_0+\dots+x_{n_0}}{n} + \frac{x_{n_0+1}+\dots+x_{n-1}}{n}\right)\pi(dx_{n_0} )\prod_{k=n_0}^{n-2}P^{\pi}_{x_k,dx_{k+1}}\\
\ge& - \rho(\delta) + \iint\dots\int f\left(\frac{\sum_{i=n_0+1}^{n-1}x_i}{n}\right)
\pi(dx_{n_0})\prod_{k=n_0+1}^{n-2}P^{\pi}_{x_k,dx_{k+1}}.
\end{flalign*}

Hence, we may obtain
\begin{flalign*}
\mathbb{E}[f(S_{n}/n)] &\le 
\rho_f(\delta) + \left(\rho_f(\delta) +  \mathbb{E}[f(S_n^\pi/n)]\right)\prod_{k=n_0-1}^{n-2}\left(1+e^{\ln K - Ck}\right).
\end{flalign*}

Analogically we can derive the following result for the lower bound of $\mathbb{E}[f(S_{n}/n)]$, that is
{\scriptsize
\begin{flalign*}
\mathbb{E}[f(S_{n}/n)]
\ge&
-\rho_f(\delta) + \iint\dots\int f\left(\frac{x_{n_0+1}+\dots+x_{n-1}}{n}\right)
\left(1-e^{\ln K - C(n_0-1)}\right)\pi(dx_{n_0})\\
&\times
\prod_{k=n_0}^{n-2}\left(1-e^{\ln K - Ck}\right)P^{\pi}_{x_k,dx_{k+1}}\\
\ge& -\rho_f(\delta) + \prod_{k=n_0-1}^{n-2}\left(1-e^{\ln K - Ck}\right)\\&\times
\left(\iint\dots\int f\left(\frac{\sum_{i=n_0+1}^{n-1}x_i}{n}\right)
\pi(dx_{n_0})\prod_{k=n_0+1}^{n-2}P^{\pi}_{x_k,dx_{k+1}}\right)\\
\ge& -\rho_f(\delta) + \left(\mathbb{E}[f(S_n^\pi/n)]-\rho_f(\delta)\right)\times \prod_{k=n_0-1}^{n-2}\left(1-e^{\ln K - Ck}\right).
\end{flalign*}}

Next, we can consider the bounds for $\mathbb{E}[f(S_n/n)]$,
{\scriptsize
\begin{flalign*}
- \rho_f(\delta) +\left(\mathbb{E}[f(S_n^\pi/n)]-\rho_f(\delta)\right) \prod_{k={n_0}-1}^{n-2}\left(1-e^{\ln K - Ck}\right) \le 
\mathbb{E}[f(S_n/n)]\\ \le
\rho_f(\delta) +\left(\rho_f(\delta) +  \mathbb{E}[f(S_n^\pi/n)]\right)\prod_{k={n_0}-1}^{n-2}\left(1+e^{\ln K - Ck}\right).
\end{flalign*}}

Notice that  $\rho_f(\delta)$ as $\delta \to 0$, and that  
both multipliers {\scriptsize$\prod_{k={n_0}-1}^{n-2}\left(1-e^{\ln K  
- Ck}\right)$} and {\scriptsize $\prod_{k={n_0}-1}^{n-2}\left(1+e^{\ln K  
- Ck}\right)$} approach $1$ as $n_0\to\infty$ uniformly  
over $n>n_0$, while $\mathbb{E}[f(S_n/n)]$ itself does not depend on $n_0$. Thus, due to \eqref{stationarylln} we obtain  
for any non-negative $f \in \mathcal{C}_b$,  
$$  
\limsup_{n\to\infty} \mathbb{E}[f(S_n/n)] \le  
f\left(\mathbb{E}[g(X_0^\pi)]\right),  
$$  
and  
$$  
\liminf_{n\to\infty} \mathbb{E}[f(S_n/n)] \ge  
f\left(\mathbb{E}[g(X_0^\pi)]\right),  
$$  
which implies  
$$  
\lim_{n\to\infty} \mathbb{E}[f(S_n/n)] =  
f\left(\mathbb{E}[g(X_0^\pi)]\right),  
$$  
which is equivalent to the convergence in probability \eqref{weaklln} as required. Theorem \ref{th3} is proved.
\end{proof}

\subsection*{Acknowledgements}
The author is grateful to Professor A. Yu. Veretennikov for his valuable and constructive advice.

\subsection*{Funding}
Russian Foundation for Basic Research, 20-01-00575.\newline http://dx.doi.org/10.13039/501100002261.

\end{document}